\documentclass[11pt,a4paper,english,reqno]{amsart}
\usepackage{amsmath,amssymb,amsfonts,epsfig,mathrsfs}
\usepackage[T1]{fontenc}

\usepackage{color}
\usepackage{array}
\usepackage{amsthm}
\usepackage{amstext}
\usepackage{graphicx}
\usepackage{setspace}
\usepackage[margin=2.5cm]{geometry}
\usepackage{bbm}
\usepackage{color}
\usepackage{enumitem}
\allowdisplaybreaks[4]

\usepackage{amscd,psfrag}
\usepackage{yhmath}
\usepackage[mathscr]{eucal}
\usepackage{slashed}
\usepackage{comment}

\makeatletter
\pdfpageheight\paperheight
\pdfpagewidth\paperwidth

\usepackage{epstopdf}
\usepackage{chngcntr}
\usepackage{mathrsfs}

\usepackage{indentfirst}	

\usepackage[normalem]{ulem}
\theoremstyle{plain}
\newtheorem{definition}{Definition}[section]
\newtheorem{theorem}[definition]{Theorem}
\newtheorem*{theorem*}{Theorem}

\newtheorem{remark}[definition]{Remark}

\newtheorem*{remark*}{Remark}
\newtheorem*{sideremark*}{Side Remark}\newtheorem*{mt*}{Main Theorem}

\newtheorem*{claim*}{Claim}
\newtheorem*{q*}{Question}

\newtheorem*{corollary*}{Corollary}
\newtheorem*{proposition*}{Proposition}
\newtheorem{example}[definition]{Example}

\newcommand{\R}{\mathbb{R}}

\newcommand{\na}{\nabla}
\newcommand{\dd}{{\rm d}}
\newcommand{\p}{\partial}
\newcommand{\e}{\epsilon}
\newcommand{\emb}{\hookrightarrow}

\newcommand{\map}{\rightarrow}
\newcommand{\G}{\Gamma}

\newcommand{\dvg}{\dd V_g}

\newcommand{\is}{{\mathscr{IS}}}

\newcommand{\E}{{\mathbf{E}}}
\newcommand{\homeo}{{\bf Homeo}}
\newcommand{\homeoo}{
\homeo_{\,0}}
\newcommand{\Z}{{\mathbb{Z}}}

\newcommand{\thomeo}{\widetilde{\homeoo}}

\newcommand{\tone}{{\mathbb{T}^1}}

\newcommand{\N}{{\mathcal{N}}}

\newcommand{\tn}{\mathbb{T}^n}

\newcommand{\volm}{{{\rm Vol}_g(M)}}

\allowdisplaybreaks[4]

\def\XXint#1#2#3{{\setbox0=\hbox{$#1{#2#3}{\int}$ }
\vcenter{\hbox{$#2#3$ }}\kern-.6\wd0}}

\newcommand{\mres}{\mathbin{\vrule height 1.6ex depth 0pt width
0.13ex\vrule height 0.13ex depth 0pt width 1.3ex}}
\newcommand{\tth}{\widetilde{\theta}}

\newcommand{\id}{{\bf id}}

\newcommand{\diff}{{\bf Diff}}
\newcommand{\sdiff}{{\bf  SDiff}}

\newcommand{\on}{{\overline{n}}}

\newcommand{\vj}{{\widetilde{V_j}}}

\title{A smooth isotopy of volume-preserving diffeomorphisms on unit cube saving energy through extra dimensions}

\author{Siran Li}

\address{Siran Li: School of Mathematical Sciences, Shanghai Jiao Tong University, No.~6 Science Buildings,
800 Dongchuan Road, Minhang District, Shanghai, China (200240)}

\email{\texttt{siran.li@sjtu.edu.cn}}



\date{\today}

\begin{document}

\maketitle

\begin{abstract}
We construct an explicit example of a smooth isotopy $\{\xi_t\}_{t \in [0,1]}$ of volume- and orientation-preserving diffeomorphisms on $[0,1]^n$ ($n \geq 3$) that has infinite total kinetic energy. This isotopy has no self-cancellation and is supported on  countably many disjoint tubular neighbourhoods of homothetic copies of the isometrically embedded image  of $(M,g)$, a ``topologically complicated'' Riemannian manifold-with-boundary. However, there exists another smooth isotopy that coincides with $\{\xi_t\}$ at $t=0$ and $t=1$ but of finite total kinetic energy. 
\end{abstract}

\section{Introduction and the main result}\label{sec: intro}

%\subsection{The Euler Equation}

In the seminal paper  \cite{a} in 1966, Arnold   interpreted the Euler equations  for incompressible inviscid flows on a Riemannian manifold-with-boundary $(M,g)$ as the geodesic equation on $\sdiff(M)$, the group of orientation-, volume-, and boundary-preserving diffeomorphisms on $M$. The Euler equation on $(M,g)$ reads as follows:
\begin{equation}\label{euler}
\begin{cases}
\p_t v + {\rm div}(v \otimes v) + \na p = 0 \qquad \text{in } [0,T] \times M,\\
{\rm div}\, v = 0\qquad \text{in } [0,T] \times M,\\
\langle v, \nu \rangle = 0 \qquad \text{on } [0,T] \times \p M.
\end{cases}
\end{equation}
Here $v: [0,T] \times M \to TM$ is the velocity,   $p: [0,T] \times M \to \R$ is the pressure, $\nu$ is the outward unit normal vectorfield on the boundary, $\langle\bullet,\bullet\rangle$ is the inner product on $\p M$ induced by the metric $g$, and ${\rm div}$ is the divergence operator with respect to the metric $g$. The infinite-dimensional Lie group $\sdiff(M)$ is equipped with the $L^2$-metric. Assume $T=1$ for normalisation.

%\subsection{Two problems on $\sdiff(M)$}\label{subsec: 2 prob}

Arnold's viewpoint has led to fruitful researches on the Euler equations via global-geometric and variational methods. One important question is the following

\begin{quote}
{\bf Two-Point Problem:} Given $h_0$ and $h_1 \in \sdiff(M)$. Find a smooth curve $\{\xi_t\}_{t \in [0,1]}$ such that  $\xi_0 = h_0$,  $\xi_1=h_1$, and 
\begin{equation}\label{J}
J(\{\xi_t\}) := \frac{1}{2} \int_0^1 \left\|\frac{\p \xi_t(\bullet)}{\p t}\right\|_{L^2}^2\,\dd t < \infty.
\end{equation}
Here $\p \xi_t /\p t\in\G(TM)$ for each $t$, the $L^2$-norm is taken with respect to the metric $g$, and $J(\{\xi_t\})$ is the \emph{action} or \emph{total kinetic energy} of the isotopy $\{\xi_t\}$.
\end{quote}

%among all the curves $[0,1]\ni t \mapsto \{\xi_t\} \subset \sdiff(M)$ satisfying $\xi_0 = h_0$ and $\xi_T=h_T$.

Without loss of generality, we take $$h_0=\id \text{ (the identity map)}$$ throughout. Under this provision, we have

\begin{definition}\label{def: attain}
We say that a diffeomorphism 
$h_1 \in \sdiff(M)$ is \emph{attainable} if  the two-point problem is soluble for $h_1$. 
\end{definition}

In the case that $h_1$ is attainable, it is natural to further consider the following

\begin{quote}
{\bf Classical Variational Problem:} Given $h_1 \in \sdiff(M)$. Find a smooth curve $\{\xi_t\}_{t \in [0,1]}$ such that $\xi_0 = \id$,  $\xi_1=h_1$, and $\{\xi_t\}$ minimises  $J(\{\xi_t\})$ in \eqref{J}. 
\end{quote}

The above problems addresses whether two prescribed configurations of an incompressible fluid on $M$ can be nicely deformed into each other inside $\sdiff(M)$. This is essentially different from the Cauchy problem for PDE. A solution for the classical variational problem will give rise to a smooth solution for the Euler Eq.~
\eqref{euler}.

%\subsection{Attainable and unattainable diffeomorphisms} 
In a series of original works \cite{s89, s93, s94} around 1990s, Shnirelman showed that  answers to both the two-point problem and the classical variational problem are, in general, \emph{negative}. More precisely, denote $I^n:=[0,1]^n$; the following  was proved in \cite{s89, s93, s94}:
\begin{theorem}\label{thm: sh}
There are unattainable diffeomorphsims on $I^2$; in contrast, any diffeomorphism $h \in \sdiff(I^n)$ is attainable when $n \geq 3$. Also, there exists an $h_{\rm BAD} \in \sdiff\left(I^{n \geq 3}\right)$ such that there are no action-minimising curves of  diffeomorphisms connecting $\id$ and $h_{\rm BAD}$. 
\end{theorem}

The existence of unattainable diffeomorphisms on $I^2$ is closely related to the fact that the diameter of $\sdiff(I^2)$ with respect to the $L^2$-metric is infinite. This extends to general surfaces via symplectic geometric techniques. See Eliashberg--Ratiu \cite{er}. On the other hand, if $h=h_1$ is sufficiently close to $\id=h_0$ in the Sobolev $H^s$-norm for $s>\lfloor n/2 \rfloor +1$, then the two-point and classical variational problems are both soluble. See Ebin--Marsden \cite{em}. This current note is essentially motivated by Theorem~\ref{thm: sh}.

%\subsection{Some subsequent developments}

%To put our work into perspectives, we digress briefly to discuss some developments further to the resolutions of the two-point and classical variational problems.  

The non-existence of action-minimising paths in $\sdiff(M)$, if it is established, suggests that the putative fluid flows which ``minimise the kinetic energy'' may allow collisions,  merges, and splits of particle trajectories. Thus, one may need to forgo the smoothness of $\sdiff(M)$ and work in measure-theoretic settings. In 1989 Brenier \cite{b} introduced the notion of \emph{generalised flows}, which are probability measures on the path space $C^0(I,M)$. A generalised flow is incompressible if and only if its marginal at any time is equal to the normalised  volume measure on $M$. Roughly speaking, it corresponds to a measure-valued solution for the Euler equations~\eqref{euler}.

Pioneered by Brenier \cite{b}, variational models for incompressible generalised flows remain up to date a central topic in mathematical hydrodynamics. Many works have been devoted to understanding the approximations by diffeomorphisms, the role of pressure, the optimal transportational aspects, and the numerical algorithms for generalised flows. We refer the readers to Bernot--Figalli--Santambrogio \cite{bfs}, Daneri--Figalli \cite{df}, Ambrosio--Figalli \cite{af}, Brenier--Gangbo \cite{bg}, Eyink \cite{e}, and Benamou--Carlier--Nenna \cite{bcn}.  In addition, Arnold's paper \cite{a} is considerably influential on the study of diffeomorphism groups. See Misio\l ek--Preston \cite{mp} and many others.

%\subsection{Fluid domains other than $I^n$ or $\tn$?}

The papers \cite{bcn, bfs, b, bg, df, e, s89, s93, s94} mentioned above on unattainable diffeomorphisms or  generalised flows all work with $I^n$ (the unit cube) or $\tn$ (the flat torus). Extensions to more general manifolds remain elusive in the large. To the knowledge of the author, the only exceptions are Ambrosio--Figalli \cite{af}, which established results for generalised flows on measure-preserving Lipschitz preimages of $I^n$ or $\tn$, and Eliashberg--Ratiu \cite{er}, which studied smooth isotopies on some classes of topologically complicated smooth manifolds. Both the formulation and the proof of our main result rely heavily on \cite{er}. This will be elaborated in subsequent sections.

On the other hand, notable developments concerning  measure-preserving homeomorphisms on manifolds-with-boundary have taken place. See, {\it e.g.}, Fathi \cite{fathi}, Schwartzman \cite{s}, and many subsequent works.

In this note, utilising results on unattainable diffeomorphisms established by Shnirelman \cite{s89, s93, s94} and exploiting the \emph{mass flow homomorphism} introduced by Fathi \cite[\S 5]{fathi}, we arrive at a somewhat surprising example of an isotopy $\{\xi_t\}_{t \in [0,1]}$; \emph{i.e.}, a smooth curve starting from $\id$ in the special diffeomorphism group $\sdiff(I^{n})$. This isotopy is obtained by glueing countably infinitely many disjoint ``building blocks'' together. Each building block is concentrated on a low-dimensional Riemannian manifold or manifold-with-boundary  $M$ that is ``topologically complicated'' as in Eliashberg--Ratiu \cite{er} and, loosely speaking, has the effect of stirring $M$ in one direction for a very large number of times $N$. The building blocks differ from each other only in $N$ up to translations and homotheties. The most important feature of the isotopy is that there is apparently no natural competing isotopy $\{\eta_t\}_{t \in [0,1]}$ with the same initial and terminal states but saves total kinetic energy/action with respect to $\{\xi_t\}_{t \in [0,1]}$; however, $J\big(\{\xi_t\}_{t \in [0,1]}\big)=\infty$, while there exists such an  $\{\eta_t\}_{t \in [0,1]}$ with $J\big(\{\eta_t\}_{t \in [0,1]}\big)<\infty$.

The only plausible energy-saving mechanism for the above construction is that, although the ``stirring'' on the low-dimensional submanifold $M$ cannot be deformed to reduce energy $J$ in $\sdiff(M)$, this ceases to be the case once we extend it to an ambient isotopy on $I^\on$, even if the extension only takes place in an arbitrarily small tubular neighbourhood of $M$. This explains the title of our paper. %Such an energy-saving phenomenon is quite ``radical'': we have $J\big(\{\xi_t\}_{t \in [0,1]}\big)=\infty$ but $J\big(\{\eta_t\}_{t \in [0,1]}\big)<\infty$. 

More precisely, our main result of the note is as follows:

\begin{theorem}\label{thm: main}
Let $I^n = [0,1]^n$ be the unit Euclidean cube; $n \geq 3$. There exists a smooth isotopy $\{\xi_t\}_{t \in [0,1]} \subset \sdiff(I^n)$ with infinite kinetic energy: $J\big(\{\xi_t\}_{t \in [0,1]}\big)=\infty$ and another smooth isotopy $\{\eta_t\}_{t \in [0,1]} \subset \sdiff(I^n)$ with $\eta_0=\xi_0=\id$, $\eta_1=\xi_1$, and $J\big(\{\eta_t\}_{t \in [0,1]}\big)<\infty$. 

In fact, let $(M,g)$ be a compact Riemannian manifold-with-boundary (possibly $\p M =\emptyset$) such that  $\dim M \geq 3$, $M$ has nontrivial first real homology, and $\pi_1(M)$ has trivial centre. Let $\iota: (M,g)\emb \E^n$ be an isometric embedding into the Euclidean space for some $n \geq \dim M$. Let $\{\rho_j\}$ be any sequence of positive numbers shrinking to zero as $j \to \infty$, such that there are disjoint Euclidean balls ${\bf B}_j \equiv {\bf B}_{\rho_j}(x_j)$ in $I^n$. Denote by $\iota_j : M \emb \frac{1}{2}{\bf B}_j\equiv {\bf B}_{\frac{\rho_j}{2}}(x_j)$ the embeddings obtained from $\iota$ via  homotheties and translations. Also, let $\{\delta_j\}$ be any sequence of positive numbers with $\delta_j < \frac{\rho_j}{4}$ such that $\delta_j$-tubular neighbourhoods of $\iota_j(M)$ are embedded. Then, the isotopy $\{\xi_t\}_{t \in [0,1]}$ can be constructed to be supported on the $\delta_j$-tubular neighbourhoods of $\iota_j(M)$.

The isotopy $\{\xi_t\}_{t \in [0,1]} \subset \sdiff(I^n)$ constructed here is ``genuinely of infinite kinetic energy'', in the sense that its restriction to each ${\bf B}_j$ never comes to a stop or return to any previous configurations at any time instant $t \in [0,1]$.
\end{theorem}

The second paragraph of Theorem~\ref{thm: main} maintains that the ``wild'' isotopy $\{\xi_t\}_{t \in [0,1]}$ can be constructed in a sufficiently localised fashion; \emph{i.e.}, supported on arbitrarily small tubular neighbourhoods of $\iota_j(M)$, the scaled embedded images of $M$ into aribitrarily small disjoint balls in $I^n$. We say that $\{\xi_t\}_{t \in [0,1]}$ is supported on a set $\mathcal{K}$ if $\xi_t \equiv \id$ outside $\mathcal{K}$ for all $t \in [0,1]$. Here the submanifold $M$ satisfies the same topological conditions as in Eliashberg--Ratiu \cite[Appendix~A]{er} --- $\mathcal{Z}\big(\pi_1(M)\big) = \{0\}$ and $H_1(M;\R)\neq \{0\}$ --- so that the mass flow (see \S\ref{sec: mass flow} below for details) can be as large as we want.

To illustrate the topological conditions on $M$, consider the following examples below. (The second is pointed out by a referee of an earlier version.)
\begin{itemize}
\item
A rope in $\E^3$, \emph{i.e.}, thickening of a knot by making a tubular neighbourhood of it solid. It has nontrivial first real homology and its fundamental group has nontrivial centre.
\item
Handlebody of genus 2: it has nontrivial first real homology, but the fundamental group is free on $2$-generators and hence has trivial centre. 
\end{itemize}

Before presenting the proof, a few crucial remarks are in order:

\begin{remark}
One may deduce from the proof of Theorem~\ref{thm: main} that each building block is energy saving, namely that
\begin{align*}
J\left(\{\eta_t\}_{t \in [0,1]}\big|_{{\bf B}_j}\right) < J\left(\{\xi_t\}_{t \in [0,1]}\big|_{{\bf B}_j}\right)\qquad \text{for every } j = 1,2,3,\ldots.
\end{align*}
This is because each building block differs from the others only in terms of $N(j)$, the number of rounds that the loop $\sigma$ traverses, in Eq.~\eqref{turns}. Moreover, the building blocks have disjoint supports from each other, so there is no reason for only some of those to be energy saving. 
\end{remark}

\begin{remark}\label{rem on bdry smoothness}
To ensure that the ``wild'' isotopy $\{\xi_t\}_{t \in [0,1]}$ lies in $\sdiff(I^n)$, the balls ${\bf B}_j$ cannot have accumulation points in the interior of $I^n$. Therefore, heuristically speaking, the construction of $\{\xi_t\}_{t \in [0,1]}$ consists of ``more and more rapid turns made on finer and finer scales'' which eventually accumulate on the boundary $\p I^n$. This is possible since an element of $\sdiff(I^n)$ is only assumed to be continuous up to the boundary.

The above convention on $\sdiff(I^n)$ is consistent with Shnirelman's results in \cite{s89, s93, s94}, especially \cite[Theorem~A]{s93} --- each element in $\sdiff(I^n)$ is attainable when $n \geq 3$ --- which is key to our proof of Theorem~\ref{thm: main}.

In fact, the proof of \cite[Theorem~A]{s93} is based upon the result by McDuff \cite{mc} on the existence of $\{\xi_t\}_{t \in [0,1]}$ which is continuous in $t$ on each subdomain $\mathcal{K}' \Subset I^n$ away from the boundary. Then, one carefully corrects the ``boundary layers'' with controllable total kinetic energy. More precisely, \cite[Lemma~1.5]{s93}  shows that for arbitrary $0<\e<\delta$, any $\xi_1 \in \sdiff(\mathcal{K_{\e,\delta}})$ can be connected to $\xi_0 = \id$ through an isotopy $\left\{\xi_t^{\e,\delta}\right\}_{t \in [0,1]}$ supported on $\mathcal{K_{\e,\delta}}$ for all time and has 
\begin{equation}\label{bdry layer}
J \left(\left\{\xi_t^{\e,\delta}\right\}_{t\in [0,1]}\right) \leq C(\delta-\e),
\end{equation}
where $C=C(n)$ is a purely dimensional constant and $\mathcal{K_{\e,\delta}}:=\left\{x \in I^n:\, \e < {\rm dist}(x, \p I^n)<\delta\right\}$. Then, as in \cite[pp.282-283, proof of Theorem~A]{s93}, one iteratively corrects on time interval $\left[1-2^k, 1-2^{k+1}\right]$ the ``boundary layer'' in $\mathcal{K}_{\e_k,\delta_k}$ with $0<\e_k<\delta_k := 4^{-k}$ for $k = 1,2,3,\ldots$. In this way, thanks to Eq.~\eqref{bdry layer}, the resulting isotopy $\{\xi_t\}_{t \in [0,1]}$ has finite total kinetic energy not only on subdomains away from the boundary as in \cite{mc}, but in the whole $I^n$. Nonetheless, note that the above construction does not ensure that $\{\xi_t\}_{t \in [0,1]}$ is smooth up to the boundary.

We also remark that Shnirelman's example of unattainable diffeomorphism on $I^2$ is also continuously but not smooth up to the boundary. See \cite[\S 2.4, Theorem~2.6]{s94}. In fact, our construction in this note is essentially motivated by the idea of this construction. 
\end{remark}

\begin{remark}
Theorem~\ref{thm: main} remains valid for any closed Riemannian surface $M$ (\emph{i.e.}, when $\dim M=2$) with $\mathcal{Z}\big(\pi_1(M)\big)=\{0\}$ and $H_1(M,\R)\neq \{0\}$, but subtleties may arise when $\dim M=2$ and $\p M \neq \emptyset$. The proof goes through verbatim, since there is no danger in selecting generators of $H_1(M,\R)$ represented by smooth embedded loops on closed surfaces.
\end{remark}

%\begin{remark}
%Theorem~\ref{thm: main} is closely related to the facts that the $L^2$-diameter of $\sdiff(M,g)$ is infinite, while that of $\sdiff(I^n)$ is finite for $n \geq 3$. Nevertheless, Theorem~\ref{thm: main} does not follow from these facts. 
%\end{remark}

\section{Notations and Nomenclatures}

The following notations shall be used throughout this note:

\begin{itemize}

\item
$I=[0,1]$ and $I^n = [0,1]^n$.

\item
For a set $S$, denote by $\mathring{S}$ its interior (with respect to an obvious topology in context).

\item
For sets $S$ and $R$, the symbol $S \sim R$ denotes the set difference. 

\item
For a group $G$, the symbol $\mathcal{Z}(G)$ denotes its centre.

\item
$(M,g)$ is a Riemannian manifold-with-boundary if $M$ is locally diffeomorphic to $\R^n$ or $\overline{\R^n_+}=\{x=(x',x^n):\,x'\in\R^{n-1}, \, x^n \geq 0\}$, and $g$ is a Riemannian metric on $M$.

\item
$\p M$ is the boundary of the manifold $M$.

\item
$\id$ denotes the identity map whose domain is clear from the context. 

\item For a topological group $X$, the symbol $X_0$ denotes the identity component of $X$.

\item
$\E^n$ is $\R^n$ equipped with the Euclidean metric.

\item
Let ${\bf B}={\bf B}_r(x)$ be a Euclidean open ball in $\E^n$. For any $\lambda>0$ write $\lambda{\bf B} := {\bf B}_{\lambda r}(x)$.

\item
$\homeo(M)$ is the group of homeomorphisms on $M$. 

\item
Let $\mu$ be a measure on $M$. $\homeo(M;\mu)$ is the $\mu$-preserving subgroup of 
$\homeo(M)$, namely that 
$$\homeo(M;\mu):=\left\{\Phi \in \homeo(M):\, \Phi_\# \mu= \mu\right\}.$$ As in \cite{fathi}, the measure $\mu$ in this note is ``good'' in the sense of Oxtoby--Ulam \cite{ou}; {\it i.e.}, $\mu$ is atomless, its support is the whole of $M$, and $\mu(\p M)=0$. 

\item
$\thomeo(M;\mu)$ is the identity component of the universal cover of $\homeo(M;\mu)$. 
\item
$\diff(M)$ is the group of diffeomorphisms on $M$. When $\p M \neq \emptyset$, an element of $\diff(M)$ also fixes the boundary (hence the corresponding vectorfields are perpendicular to $\p M$).

\item
$\Phi_\#$ and $\Phi^\#$ denote the pushforward and pullback under $\Phi$, respectively. They can be defined for suitable measures, maps, tensorfields, etc.

\item
$\dvg$ is the Riemannian volume measure on $M$; $\Omega$ is the normalised Riemannian volume measure, \emph{i.e.}, $\Omega = \left[{\rm Vol}_g(M)\right]^{-1}\dvg$. 
\item
$\mathcal{L}^n$ is the $n$-dimensional Lebesgue measure. 
\item
$\sdiff(M):=\{\Phi \in \diff(M):\, \Phi_\# \Omega = \Omega\}$.

\item
A path $\{\xi_t\}_{t \in I}$ in $\sdiff(M)$ with $\xi_0=\id$ is known as a smooth \emph{isotopy} connecting $\id$ to $\xi_1$. We also write it as $\{\xi_t\}$. In Fathi \cite[Appendix~A.5]{fathi} it is written as $(\xi_t)$, and the group of isotopies is denoted by $\mathscr{IS}^\infty(M)$. Moreover, $\mathscr{IS}(M;\mu)$ denotes the subgroup of $\mu$-preserving smooth isotopies. Therefore, the two-point problem in \S\ref{sec: intro} concerns the existence of smooth isotopies with volume measure-preserving prescribed endpoints. 

\item
Say that $\{\xi_t\}_{t \in [0,1]}$ is supported on a set $\mathcal{K}$ if $\xi_t \equiv \id$ outside $\mathcal{K}$ for all $t \in [0,1]$.

\item
$J\{\xi_t\}$ denotes the action (total kinetic energy) of the isotopy $\{\xi_t\}$, defined in Eq.~\eqref{J}.

\item
$H_k(M;\R)$ and $H^k(M;\R)$ denote the $k$-th $\R$-coefficient homology and cohomology groups of $M$, respectively.

\item
$\pi_1(M)$ is the fundamental group of $M$.

\item
For a vectorfield $X$ and a differential form $\beta$, the interior multiplication of $X$ to $\beta$ is written as $X\mres \beta$.

\item
Unless otherwise specified, $[\bullet]$ designates an equivalence class. For example, for a closed $k$-form $\beta$ on $M$, $[\beta]$ denotes the corresponding cohomology class in $H^k(M;\R)$. 

\item
$\tone = \R\slash 2\pi\mathbb{Z} = \mathbb{S}^1$. We write $\tone$ to emphasise that the group operation is additive.

\item
$\left[M,\tone\right]$ denotes the homotopy class of maps $M\to\tone$.

\item
For $f: X \map \tone$ with $f(0)=0$,  write $\overline{f}$ for its lift to the universal cover $\R$ with $\overline{f}(0)=0$. 

\item
We reserve the symbol 
$$\theta: \homeo_0(M,\mu) \longrightarrow H_1(M;R)$$ for the mass flow homomorphism on $M$. See \S \ref{sec: mass flow} for its definition.

\item
$\tth$ is the lift of $\theta$. See \S \ref{sec: mass flow} too for details.
\end{itemize}

\begin{comment}
It is important to point out the following.
\begin{remark}
In the definition of $\diff(M)$ and $\sdiff(M)$ for a manifold-with-boundary $M$, we do NOT require their elements to be smooth up to the boundary. Instead, we only require interior smoothness and continuity up to the boundary. This agrees with Shnirelman \cite{s94}. %The construction of $S_\infty$ in Eq.~\eqref{s inf}.
\end{remark}
\end{comment}

\section{Mass flow}\label{sec: mass flow}

An essential ingredient of our developments is the mass flow homomorphism defined by Fathi \cite[\S 5]{fathi}. The idea had also appeared in Schwartzman \cite{s} under the name of ``asymptotic cycles''. We summarise here the definition and some properties of the mass flow.%, which will be needed for the proof of Theorem~\ref{thm: main} in \S\ref{sec: proof}. 

Let $M$ be a compact metric space. Given any good measure $\mu$ on $M$ (in the  sense of Oxtoby--Ulam \cite{ou}), one can define a natural map
\begin{equation}\label{tilde-massflow}
\tth: \thomeo(M;\mu) \longrightarrow {\rm Hom}\left(\left[M,\tone\right];\R\right),
\end{equation}
where $\left[M,\tone\right]$ is the homotopy class of maps from $M$ to $\tone$. When $M$ is a compact manifold-with-boundary, we have $${\rm Hom}\left(\left[M,\tone\right];\R\right)\cong H_1(M;\R),$$ as well as
\begin{equation}\label{h as isotopy}
 \thomeo(M;\mu) = \is(M;\mu) \slash \sim,
\end{equation} 
with the equivalence relation $\{h_t\} \sim \{k_t\}$ if and only if $h_0=k_0=\id$, $h_1=k_1$, and there is a continuous map $F: I \times I \to \homeo(M;\mu)$ such that $F_{0,s}=\id$, $F_{1,s}=h_1$, $F_{t,0}=h_t$, and $F_{t,1}=k_t$. The space $\is(M;\mu) \slash \sim$ is equipped with the quotient topology derived from the compact-open topology on  the isotopy group. 

To define $\tth$ in Eq.~\eqref{tilde-massflow}, let us take any $\{h_t\} \in \is(M;\mu)$. Then, for any $f \in C^0\left(M,\tone\right)$, we have a homotopy $(f\circ h_t - f) : M \map \tone$ which equals $0$ at $t=0$. So one may lift it in a unique way to a function $\overline{f \circ h_t - f}: M \map \R$ with $\overline{f \circ h_0 - f}=0$. We set
\begin{equation}\label{tilde theta, def}
\tth(\{h_t\}) (f) := \int_M \overline{f \circ h_1 - f}\,\dd\mu.
\end{equation}
One can easily check that $\tth$ induces a homomorphism as in Eq.~\eqref{tilde-massflow}, and that $\tth$ is continuous if ${\rm Hom}\left(\left[M,\tone\right];\R\right)$ is endowed with the weak topology.

Write $\N(M;\mu)$ for the kernel of the universal cover $\thomeo(M;\mu) \to \homeo_0(M;\mu)$. It is identified with the loop space:
\begin{align*}
\N(M;\mu) &\cong \Big\{ \big[ \{h_t\} \big] \in \is(M;\mu)\slash \sim :\, h_1=\id \Big\}\\
&\cong \Big\{ \{h_t\}  \in \thomeo(M;\mu):\, h_1=\id \Big\}.
\end{align*}
If $M$ is a compact manifold-with-boundary, we further set
\begin{equation}\label{Gamma}
\G := \tth\left( \N(M;\mu)\right) \leq H_1(M;\R).
\end{equation}
Then $\tth$ declines to a homomorphism
\begin{equation}\label{theta, massflow}
\theta: \homeo_0(M;\mu) \longrightarrow H_1(M;\R) \slash \G.
\end{equation}
We shall call either $\tth$ in Eq.~\eqref{tilde-massflow} or $\theta$ in Eq.~\eqref{theta, massflow} the \emph{mass flow homomorphism}.

The integrand $\overline{f \circ h_1 - f}$ on the right-hand side of Eq.~\eqref{tilde theta, def} should be understood as $\lim_{t \nearrow 1}\left(\overline{f \circ h_t - f}\right)$. For $f: M \to \tone$ with $M$ connected, $\overline{f \circ h_1 - f}$ (where $h_1 = \id$) can potentially take any constant integer as its value. Thus $\G$ in Eq.~\eqref{Gamma} is nontrivial in general.

When $M$ is a compact Riemannian manifold-with-boundary equipped with a normalised volume form $\Omega$, namely that $\int_M \Omega=1$, one may further characterise $\tth$ via smooth isotopies. (Fathi \cite{fathi} attributed this result to W. Thurston, ``\emph{On the structure of the group of volume preserving diffeomorphism} [to appear]''. It seems to the author that this  paper has never appeared in publication, though.) Indeed, $\tth$ is equivalent via the Poincar\'{e} duality to 
\begin{equation*}
\widetilde{V} : \is^\infty(M;\Omega) \longrightarrow H^{n-1}(M;\R),
\end{equation*}
where for $\{h_t\} \in \is^\infty(M;\Omega)$ we define $\widetilde{V} (\{h_t\})$ as follows: let $X_t \in \G(TM)$ be the vectorfield whose integral flow is $h_t$. Then 
\begin{equation*}
\widetilde{V} (\{h_t\}) := \int_0^1 X_t \mres \Omega \,\dd t,
\end{equation*}
where $\mres$ denotes the interior multiplication.

Finally, let us present \cite[Proposition~5.1]{fathi}, attributed therein to D. Sullivan. Consider the homomorphisms:
\begin{eqnarray*}
&&a_M: \left[M,\tone\right] \longrightarrow H^1(M;\Z),\qquad f \longmapsto f^\#\sigma;\\
&&b_M: H_1(M;\Z) \longrightarrow {\rm Hom}\left(H^1(X;\Z);\Z\right),\qquad\text{the natural duality};\\
&&h_M: \pi_1(M) \longrightarrow H_1(M;\Z),\qquad\text{the Hurewicz map},
\end{eqnarray*}
and define
\begin{equation*}
\gamma_M :=  a_M^* \circ  b_M \circ h_M.
\end{equation*}
Then, whenever $M$ is a connected compact metric space, it holds that
\begin{equation}\label{sullivan}
\G \subset \gamma_M\Big(\mathcal{Z}\big(\pi_1(M)\big)\Big).
\end{equation}

\section{Proof of the main Theorem~\ref{thm: main}}\label{sec: proof} 

Now we are ready to prove the main theorem by explicitly constructing the infinite-energy isotopy $\{\xi_t\}_{t \in I} \subset \sdiff(I^n)$ and deducing the existence of finite-energy isotopy $\{\eta_t\}_{t \in I}$ from Shnirelman's theorem on attainable diffeomorphisms (\cite{s93}). As commented prior to the statement of the theorem, our proof consists of  three key components:
\begin{enumerate}
\item
the mass flow homomorphism in \S\ref{sec: mass flow};
\item
the idea of Shnirelman's proof for the existence of an unattainable diffeomorphism on $I^2$ (\emph{cf.} Theorem~\ref{thm: sh} and \cite{s94}); and
 \item
lower bound for total kinetic energy $J$ by the mass flow (\emph{cf.} Eliashberg--Ratiu \cite[Appendix~A]{er}).
\end{enumerate}

\begin{proof}[Proof of  Theorem~\ref{thm: main}]

We divide our arguments into five steps below.

\smallskip
\noindent
{\bf Step~1.} Let $(M,g)$ be a compact Riemannian manifold-with-boundary (possibly with $\p M = \emptyset$) such that
\begin{align*}
\mathcal{Z}\big(\pi_1(M)\big)=\{0\},\quad H_1(M,\R) \neq \{0\}, \quad \text{ and } m=\dim M \geq 3.
\end{align*}
Without loss of generality, assume that $M$ is connected. Also let
\begin{equation*}
\iota: (M,g) \emb \E^{{n}}
\end{equation*}
be a smooth isometric embedding into Euclidean space, whose existence is guaranteed by Nash's embedding theorem.

Consider the mass flow homomorphism in Eq.~\eqref{theta, massflow}: $$\theta: \homeoo(M,\mu) \to H_1(M;\R)\slash\G,$$ where $\mu$ is the normalised volume measure on $(M,g)$ (which is clearly ``good'' in the sense of Oxtoby--Ulam \cite{ou}) and $\homeoo(M,\mu)$ is the identity component of the group of $\mu$-preserving homeomorphisms. Under the assumption that $\pi_1(M)$ has trivial centre, $\G$ is trivial in view of Eq.~\eqref{sullivan}. In fact,  $\theta$ is a group epimorphism (\cite[\S 5]{fathi}).

For a closed manifold $M$, we may choose generators of the first real homology group $H_1(M;\R)$ to be smooth loops with smoothly embedded tubular neighbourhoods. When $M$ is a compact manifold-with-boundary, this can also be done with tubular neighbourhoods uniformly away from the boundary, thanks to the product structure (\emph{i.e.}, the ``collar'') near $\p M$.

Fix a generator $[\sigma] \in H_1(M;\R)$ as in the previous paragraph. For each $j\in \mathbb{N}$ take $h_j \in \homeoo(M,\mu)$ such that 
\begin{equation}\label{turns}
\theta(h_j) = N(j)[\sigma],
\end{equation}
where $N(j)$ is a large natural number to be specified later. Thus $\|\theta(h_j)\|= c N(j)$, where the constant $c$ depends only on the choice of $[\sigma]$. Here and hereafter, $\|\bullet\|$ denotes the quotient norm on cohomology/homology groups induced by the usual $L^2$-norm on differential forms.

Moreover, by virtue of Eq.~\eqref{h as isotopy}, via lifting to the universal cover of $\homeoo(M,\mu)$ one may identify each $h_j$ with an isotopy on $(M,g)$, denoted as $\left\{h_{j,t}\right\}_{t\in I}$.

\smallskip
\noindent
{\bf Step~2.}
To proceed, let us identify $M$ with its isometrically embedded image $\iota(M)$ in the sequel. Say that $M \subset {\bf B}_R(0) \subset \E^n$. For each $j \in \mathbb{N}$ there exists $O_j$, a composition of homotheties and Euclidean rigid motions on $\E^n$, such that
\begin{align*}
O_j(M) \subset \frac{1}{2}{\bf B}_j \equiv {\bf B}_{\frac{\rho_j}{2}}(x_j) \subset I^n.
\end{align*}
The balls ${\bf B}_j$ are as in the statement of the theorem; in particular, they are disjoint from each other. Note that the Jacobian of $O_j$ is given by
\begin{align}\label{jacobian of Oj}
{\bf Jac}(O_j) = \left(\frac{\rho_j}{2R}\right)^n.
\end{align}

Recall that $\left\{h_{j,t}\right\}_{t\in I}$ obtained above is an isotopy on $M$ for each $j$. Now let us extend it to an ambient isotopy $\{\mathcal{H}_{j,t}\}_{t\in I} \subset \sdiff(I^n)$. Suppose that $M\subset {\bf B}_R(0) \subset \E^n$ has an smoothly embedded $\delta$-tubular neighbourhood $\mathcal{T}_\delta(M)$, and denote by $\pi$ the nearest point projection from $\mathcal{T}_\delta(M)$ to $M$. Also recall that for each $j$ the vectorfield
\begin{align}\label{Vj, new}
V_j := \frac{\p h_{j,t}}{\p t} \circ \left(h_{j,t}\right)^{-1} \in \G(TM)
\end{align} 
is divergence-free on $M$, since $\left\{h_{j,t}\right\}_{t \in I} \subset \sdiff(M)$. 

The extension $\{\mathcal{H}_t\}_{t\in I}$ can be given by defining a divergence-free vectorfield $\vj \in \G\left(T\mathcal{T}_\delta(M)\right) $ which extends $V_j$.  For this purpose, choose a cutoff function $\eta \in C^\infty\big( [0,\infty[, [0,1]\big)$ such that
\begin{equation*}
\eta(s) = \begin{cases}
1\qquad\qquad\qquad \text{ for } 0 \leq s \leq \frac{\delta}{2},\\
{\rm decreasing }\qquad\text{ for } \frac{\delta}{2} < s < \frac{3\delta}{4},\\
0\qquad\qquad\qquad \text{ for } \frac{3\delta}{4} \leq s.
\end{cases}
\end{equation*}
Consider the geodesic normal coordinate on $\mathcal{T}_\delta(M)$: a generic point can be written as $x = x_h + x_v \equiv \left[x_1, \ldots, x_m\right]^\top + \left[x_{m+1}, \ldots, x_{n+m}\right]^\top$, where each coordinate component of the vertical component $x_v$ lies in $[-\delta,\delta]$, and $x_h \in M$ is the horizontal component of $x$. Then, for any vectorfield $W \in \G\left(T\mathcal{T}_\delta(M)\right) $ we have $W = W_h + W_v$ with $W_h = \pi_\# W = \dd \pi (W)$. In the geodesic normal coordinate we may express the Euclidean metric as
\begin{align*}
g_{\rm E} = \begin{bmatrix}
g &0\\
0& I_{(n-m) \times (n-m)}
\end{bmatrix}.
\end{align*}
Thus, the Euclidean divergence of $\eta W$ can be computed in local coordinates via
\begin{align*}
{\rm div} (\eta W) &= \sum_{j=1}^{n} \frac{1}{\sqrt{\det g_{\rm E}}} \p_j \left(\sqrt{\det g_{\rm E}} \,\eta W^j\right)\\
&= \sum_{j=1}^{n} \frac{1}{\sqrt{\det g}} \p_j \left(\sqrt{\det g} \,\eta W^j\right)\\
&=\eta\, {\rm div}_g ( W_h) + \sum_{j=1}^n W^j \p_j\eta + \sum_{j=n-m+1}^{n} \Big\{\p_j \left(\eta W^j\right) + \p_j \left(\log \det g\right) \eta W^j\Big\}\\
&= \sum_{j=n-m+1}^{n} \Big\{\p_j \left(\eta W^j\right) + \p_j \left(\log \det g\right) \eta W^j\Big\},
\end{align*}
thanks to the divergence-free property of $W_h$. We adopt the slight abuse of notations:
\begin{align*}
\eta(x) \equiv \eta\big(|{x_v}|_\infty\big) = \eta \Big(\max_{1\leq i \leq n} \big|x_{m+i}\big|\Big)
\end{align*}
Here we may work with the Euclidean metric $g_{\rm E}$ since $\iota$ is an \emph{isometric} embedding.

The above computations show that, given the horizontal divergence-free vectorfields $V_j$ as in Eq.~\eqref{Vj, new}, the vectorfield
\begin{align*}
\vj(x_h+x_v) := \eta (x_v) V_j (x_h)\qquad \text{ for each $j=1,2,3,\ldots$} 
\end{align*}
is divergence-free and compactly supported on the tubular neighbourhood $\mathcal{T}_\delta(M)$. Define the ambient isotopy $$\{\mathcal{H}_{j,t}\}_{t\in I} \in \sdiff(\E^n)$$ via
\begin{align*}
\vj = \frac{\p \mathcal{H}_{j,t}}{\p t} \circ \left(\mathcal{H}_{j,t}\right)^{-1} \in \G\left(T\mathcal{T}_\delta(M)\right).
\end{align*}
Then $\{\mathcal{H}_{j,t}\}_{t\in I}$ is an extension of $\{h_{j,t}\}_{t\in I}$ supported in the tubular neighbourhood $\mathcal{T}_\delta(M)$.

\smallskip
\noindent
{\bf Step~3.} The total kinetic energy/action of $\{\mathcal{H}_{j,t}\}_{t\in I} \in \sdiff(\E^n)$ can be easily estimated from below. Indeed, using the definition of $\vj$ and the choice of $\eta$, we have 
\begin{align}\label{action, big H}
J\big(\{\mathcal{H}_{j,t}\}_{t\in I}\big) &= \frac{1}{2} \int_0^1 \left\|\vj(t)\right\|^2_{L^2(\mathcal{T}_\delta(M),\, g_{\rm E})}\,\dd t\nonumber\\
&= \frac{1}{2}\int_0^1 \int_{[-\delta,\delta]^{n-m}} \eta^2(x_v) \int_{M} \left|V_j(x_h)\right|_g^2 \sqrt{\det g}\,\dd x_h \,\dd x_v\nonumber\\
&\geq \frac{1}{2} \delta^{n-m}J\big(\{h_{j,t}\}_{t\in I}\big).
\end{align}
The final inequality holds because the nonnegative function $\eta(x_v) \equiv 1$ for $x_v \in \left[-\frac{\delta}{2},\,\frac{\delta}{2}\right]^{n-m}$.

Now let us scale $\{\mathcal{H}_{j,t}\}_{t\in I}$ via the maps $O_j$ introduced above; choose $\delta_j$ as in the statement of the theorem in lieu of $\delta$ for the width of the tubular neighbourhood. We thus obtain on each ${\bf B}_j$ an isotopy $\{{O_j}_\#\mathcal{H}_{j,t}\}_{t\in I}$ that is compactly supported in $\mathcal{T}_{\delta_j}\left(\frac{1}{2}{\bf B}_j\right) \subset \frac{3}{4}{\bf B}_j$, because $\delta_j < \frac{\rho_j}{4}$ where $\rho_j$ is the radius of ${\bf B}_j$.  Since $O_j$ is a Euclidean homothety modulo rigid motions, we deduce from Eq.~\eqref{jacobian of Oj} that
\begin{align}\label{action, Oj H}
J\Big(\big\{{O_j}_\#\mathcal{H}_{j,t}\big\}_{t\in I}\Big) &= \left[{\bf Jac}(O_j)\right]^2 J\big(\{\mathcal{H}_{j,t}\}_{t\in I}\big)\nonumber\\
&= \left(\frac{\rho_j}{2R}\right)^{2n} J\big(\{\mathcal{H}_{j,t}\}_{t\in I}\big)\qquad \text{for each $j=1,2,3,\ldots$}.
\end{align}

Recall that $\left\{{\bf B}_j\right\}$ are disjoint open balls in $I^n$. We are now at the stage of defining the isotopy $\{\xi_t\}_{t \in I}$ by glueing the above ${O_j}_\#\mathcal{H}_{j,t}$: 
\begin{equation}\label{def, xi_t, new}
\xi_t := \begin{cases}
{O_j}_\#\mathcal{H}_{j,t}\qquad \text{ on each } {\bf B}_j,\\
\id \qquad \qquad \text{on } I^n \sim \left(\bigsqcup_{j=1}{\bf B}_j \right).
\end{cases}
\end{equation}
This together with Eqs.~\eqref{action, Oj H} and \eqref{action, big H} yields that
\begin{align}\label{J of xi, new}
J\big(\{\xi_t\}_{t\in I}\big) &= \sum_{j=1}^\infty J\Big(\big\{{O_j}_\#\mathcal{H}_{j,t}\big\}_{t\in I}\Big) \nonumber\\
&\geq \sum_{j=1}^\infty \frac{1}{2} \left(\frac{\rho_j}{2R}\right)^{2n} (\delta_j)^{n-m}J\big(\{h_{j,t}\}_{t\in I}\big).
\end{align}

\smallskip
\noindent
{\bf Step~4.} In this step, we show that the isotopy $\{h_{j,t}\}_{t\in I} \subset \sdiff(M,g)$ can be chosen to have arbitrarily large total energy. This essentially follows from Eliashberg--Ratiu \cite[Appendix~A]{er}.

For this purpose, recall from \S\ref{sec: mass flow}:
\begin{quote}
The group homomorphism sending a smooth isotopy to a representitive of the $(m-1)^{\text{th}}$-cohomology:
\begin{align}\label{poincare dual}
\left\{h_{j,t}\right\}_{t\in I} \quad \longmapsto\quad \left[\int_0^1 \left\{  \frac{\p h_{j,t}}{\p t} \mres \Omega\right\} \,\dd t  \right] \quad \in H^{m-1}(M;\R)
\end{align}
is the Poincar\'{e} dual of $\tth\left(\{h_{j,t}\}_{t \in I}\right)$.
\end{quote}
Here $\tth$ is the lift of the mass flow homomorphism $\theta$, $m = \dim M$, and the square bracket denotes an equivalence class in the cohomology group. Since $\pi_1(M)$ has trivial centre, we may work with $\theta$ and $\tth$ interchangeably, thanks to \cite[Proposition~5.1]{fathi} (attributed to D. Sullivan).

Now, recall from Eq.~\eqref{turns} that the $\theta$-image of $\{h_{j,t}\}_{t \in I} \in \sdiff(M,g)$ is $N(j) [\sigma]$, where $[\sigma]$ is the nontrivial cycle representative of $H_1(M;\R)$ chosen as before. By the definition of mass flow homomorphism (see Eqs.~\eqref{tilde-massflow} and \eqref{tilde theta, def}), it holds that
\begin{align*}
\tth\left(\left\{h_{j,t}\right\}_{t \in I}\right)(\varphi) = \int_{M} \overline{\varphi \circ h_{j,1} - \varphi} \,\dvg\qquad \text{for each test function $\varphi \in C^0\left(M,\tone\right)$}.
\end{align*}
Here $\overline{\varphi \circ h_{j,1} - \varphi}$ is the $t \nearrow 1$ limit of the lift of ${\varphi \circ h_{j,t} - \varphi}$ from $\tone$ to the universal cover $\R$ with $\overline{\varphi\circ h_{j,0}-\varphi}=0$. Hence, by virtue of the choice of $\{h_{j,t}\}_{t\in I}$, we have
\begin{align}\label{new, lower bound for N}
\left\|
\tth\left(\left\{h_{j,t}\right\}_{t \in I}\right)\right\| \geq c N(j),
\end{align}
where $c$ is a constant depending only on the choice of $[\sigma]$. We have not yet specified the positive integers $N(j)$.

Now we can relate Eq.~\eqref{new, lower bound for N} to the total kinetic energy of $\{\xi_t\}_{t \in I}$ defined in Eq.~\eqref{def, xi_t, new}. This is achieved by way of the identity
\begin{align}\label{new, id}
\left\|\frac{\p h_{j,t}}{\p t} \mres \Omega\right\|^2 := \int_M \left(\frac{\p h_{j,t}}{\p t} \mres \Omega\right) \wedge \star \left(\frac{\p h_{j,t}}{\p t} \mres \Omega\right) = \int_M \left|\frac{\p h_{j,t}}{\p t}\right|^2_g\,\Omega,
\end{align}
where $\star$ is the Hodge star with respect to the metric $g$, and $\Omega$ is the normalised Riemannian volume form, namely that $\Omega = \dvg\slash{\rm Vol}_g(M)$. Indeed, we have
\begin{align}\label{Jh, lower bd}
J\big( \{h_{j,t}\}_{t\in I}\big) &:= \frac{1}{2}\int_0^1 \int_M \left|\frac{\p h_{j,t}}{\p t}\right|^2\,\dvg\,\dd t\nonumber\\
&= \frac{\volm}{2} \int_0^1 \left\|\frac{\p h_{j,t}}{\p t} \mres \Omega\right\|^2\,\dd t\nonumber\\
&\geq  \frac{\volm}{2} \left\{\int_0^1 \left\|\frac{\p h_{j,t}}{\p t} \mres \Omega\right\|\,\dd t\right\}^2\nonumber\\
&\geq \frac{\volm}{2} \left\|
\tth\left(\left\{h_{j,t}\right\}_{t \in I}\right)\right\|^2\nonumber\\
&\geq c\volm N(j)^2,
\end{align}
where $c$ depends only on $[\sigma] \in H_1(M,\R)$. In the above, the second line follows from Eq.~\eqref{new, id}, the third line holds by Jensen's inequality, the fourth line follows from the surrounding remarks about Eq.~\eqref{poincare dual}, and the final line can be deduced from Eq.~\eqref{new, lower bound for N}.

Therefore, we now infer from Eqs.~\eqref{J of xi, new} and \eqref{Jh, lower bd} that
\begin{align}\label{final lower bound}
J\big(\{\xi_t\}_{t\in I}\big) \geq  c \volm  \sum_{j=1}^\infty \left(\frac{\rho_j}{2R}\right)^{2n} \left({\delta_j}\right)^{n-m} N(j)^2.
\end{align}   
Again, $c$ depends only on $[\sigma]$. Thus, by choosing for each $j \in \mathbb{N}$ the positive integer $N(j)$ sufficiently large that depends only on $\rho_j$, $\delta_j$, $m=\dim M$, and $n$ --- for instance,
\begin{align*}
N(j) =  \lceil {j^{-\frac{1}{2}} \left(\rho_j\right)^{-2n} \left(\delta_j\right)^{m-n}} \rceil,
\end{align*} 
we deduce from Eq.~\eqref{final lower bound} that
\begin{align*}
J\big(\{\xi_t\}_{t\in I}\big) = \infty.
\end{align*}

\smallskip
\noindent
{\bf Step 5.} It remains to show the existence of another isotopy $\{\eta_t\}_{t \in [0,1]} \subset \sdiff(I^n)$ with 
\begin{align*}
\eta_0=\id,\qquad \eta_1=\xi_1,\qquad \text{ and } J\big(\{\eta_t\}_{t \in [0,1]}\big)<\infty.
\end{align*}
This immediately follows from Theorem~\ref{thm: sh} that $\xi_1 \in \sdiff(I^n)$ is an attainable diffeomorphism. See Definition~\ref{def: attain}, Remark~\ref{rem on bdry smoothness}, and Shnirelman \cite{s93}.

  The proof is now complete.  \end{proof}

\section{Concluding Remarks}\label{sec: remarks}

In this note, we constructed explicitly a smooth isotopy $\{\xi_t\}_{t \in I}$ of volume-preserving and orientation-preserving diffeomorphism on $I^n$ for $n \geq 3$ with infinite total kinetic energy/action. This isotopy has no self-cancellations --- in each connected component of its support (a subset of ${\bf B}_j$), the vectorfield of $\{\xi_t\}_{t \in I}$ never comes to a stop, nor does the ``fluid'' domain return to any previous configurations at any time instant $t \in [0,1]$. Such $\{\xi_t\}_{t \in I}$ is concentrated near homothetic copies of the  isometrically embedded image of a topologically complicated manifold-with-boundary $M$, smeared out in  thin tubular neighbourhoods, wherein the vectorfields associated to $\{\xi_t\}_{t \in I}$ make extremely rapid turns. Explicitly computations in \S\ref{sec: proof} yield that $J\big(\{\xi_t\}_{t \in I}\big)=\infty$.

On the other hand, as $\xi_1$ is an  attainable diffeomorphism by Shnirelman \cite{s93}, one can always make a ``short-cut'' by connecting $\xi_0=\id$ to $\xi_1$ via a finite-energy path inside $\sdiff(I^n)$. The energy-saving mechanism lies in the lift from the isotopy on submanifold $M^m$ to the ambient isotopy on $I^n$. Short-cuts arise from the extra $n-m$ dimensions.

In the end, let us take the opportunity to correct a flawed theorem in an earlier version \cite{early} of this note. In \cite[Definition~1.1]{early} we introduced the following notion. Let $(M^m,g)$ be a compact Riemannian manifold-with-boundary as in the current note. We say that an isometric embedding $\iota: (M,g) \emb \E^n$ has the ``rigid isotopy extension property'' if and only if for every $\{\xi_t\}_{t \in I} \subset \sdiff(M,g)$ and for any $\e>0$, the following two conditions hold:
 \begin{itemize}
 \item
 There exists $\left\{\hat{\xi}_t\right\}_{t\in I} \subset \sdiff\left(\E^{{n}}\right)$ such that $
J\big(\{\xi_t\}_{t\in I}\big) + \e > J\left(\left\{\hat{\xi}_t\right\}_{t\in I}\right)$ and 
$\hat{\xi}_t \circ \iota = \iota \circ \xi_t$ for all $t \in I$.
\item
For any $\left\{\check{\xi}_t\right\}_{t\in I} \subset \sdiff\left(\E^{{n}}\right)$ such that  
$\check{\xi}_t \circ \iota = \iota \circ \xi_t$ for all $t \in I$, one must have  $
J\left(\{\xi_t\}_{t\in I}\right) < J\left(\left\{\check{\xi}_t\right\}_{t\in I}\right) + \e$.
 \end{itemize}
Then we claimed in \cite[Theorem~1.4]{early} a dynamical/topological obstruction on $M$ for the \emph{existence} of such an isometric embedding $\iota$:
\begin{quote}
Suppose that $(M,g)$ has nontrivial first real homology and its fundamental group has trivial centre. Then any of its isometric embeddings into the Euclidean space $\E^n$ with $n \geq 3$ cannot have the rigid isotopy extension property.
\end{quote}

This theorem, albeit true, is vacuously true: the rigid isotopy extension property defined above can \emph{never} hold! The reason is already manifested in the proof of main Theorem~\ref{thm: main} ---  we can always choose an ambient isotopy, which costs arbitrarily small extra kinetic energy than $\{\xi_t\}_{t \in I}$, by repeating the construction in Step~2 of the proof of Theorem~\ref{thm: main} with very thin tubular neighbourhoods. We are deeply indebted to an anonymous referee of \cite{early} for pointing out this issue.

\bigskip
\noindent
{\bf Acknowledgement}. The research of SL is supported by National Natural Science Foundation of China Project $\#$12201399,  SJTU-UCL joint seed fund WH610160507/067, and the Shanghai Frontier Research Institute for Modern Analysis. The author is greatly indebted to Professor Alexander Shnirelman for very insightful discussions on his works \cite{s93, s94}, when the author undertook a CRM--ISM postdoctoal fellowship at the Centre de Recherches Math\'{e}matiques and the Institut des Sciences Math\'{e}matiques in Montr\'{e}al during 2017--2019. SL also extends sincere gratitude to the anonymous referees of earlier versions of the manuscript for pointing out the example of genus-2 handlebody and for illustrating the extension of isotopies in small tubular neighbourhoods. SL also thanks Professors László Székelyhidi Jr. and Yanyi Xu for valuable comments, and Professor Shengkui Ye for kindly teaching me topology of smooth manifolds.

\end{document}